\documentclass[11pt]{amsart}
\usepackage{amsmath}
\usepackage{cases}
\usepackage{amssymb, verbatim}
\usepackage{comment}

\title[A special Lagrangian type equation for holomorphic line bundles]{A special Lagrangian type equation for holomorphic line bundles}
\author[A. Jacob]{Adam Jacob*}
\thanks{$^{*}$Supported in part by NSF Grant No. DMS-1204155. .}
 \address{Department of Mathematics, Harvard University, 1 Oxford St., Cambridge, MA 02138}
 \email{ajacob@math.harvard.edu}
\author[S.-T. Yau]{Shing-Tung Yau}
 \email{yau@math.harvard.edu}

\theoremstyle{plain}
\newtheorem{thm}{Theorem}[section]
\newtheorem{prop}[thm]{Proposition}
\newtheorem{defn}[thm]{Definition}
\newtheorem{lem}[thm]{Lemma}

\theoremstyle{definition}

\numberwithin{equation}{section}

\renewcommand{\thanks}[1]{\footnote{#1}} % Use this for footnotes

\newcommand{\be}{\begin{equation}}
\newcommand{\bea}{\begin{eqnarray}}
\newcommand{\eea}{\end{eqnarray}} \newcommand{\ee}{\end{equation}}
 
 \def\ba{\begin{eqnarray}}
\def\ea{\end{eqnarray}}

\def\C{{\bf C}}

\def\ra{\rightarrow}

\def\o{\omega}

\def\o{\omega}

\def\al{\alpha}
\def\b{\beta}

\def\d{\delta}

\def\o{\omega}

\def\t{\theta}

\def\ti{\tilde}

\def\Z{{\bf Z}}

\def\R{{\bf R}}
\def\C{{\bf C}}

\def\ra{\rightarrow}

\def\[{{\bf [}}
\def\]{{\bf ]}}

\def\pl{\partial}

\begin{document}
\maketitle

\begin{abstract}
Let $L$ be a holomorphic line bundle over a compact K\"ahler manifold $X$. Motivated by mirror symmetry, we study the deformed Hermitian-Yang-Mills equation on $L$, which is  the line bundle analogue of the special Lagrangian equation in the case that $X$ is Calabi-Yau. We show that this equation is the Euler-Lagrange equation for a positive functional, and that solutions are unique global minimizers. We provide a necessary and sufficient criterion for existence in the case that $X$ is a K\"ahler surface. For the higher dimensional cases, we introduce a line bundle version of the Lagrangian mean curvature flow, and prove convergence when $L$ is ample and $X$ has non-negative orthogonal bisectional curvature. 

\end{abstract}

\begin{normalsize}

\section{Introduction}

At the most fundamental level, mirror symmetry describes a framework for relating complex geometry to symplectic geometry on two Calabi-Yau manifolds. It defines a duality between the underlying structures on each manifold, and the development of this elegant theory has lead to progress in both physics and mathematics. One particular aspect of mirror symmetry we are interested in is the relationship between the derived category of coherent sheaves on one manifold, and Fukaya's category of Lagrangian submanifolds with local systems on the other \cite{K,KS}. 
%Developments in this direction... [cite]

This duality can be approached from a differential geometric perspective \cite{GW, SYZ}. In the simple case of a torus, drawing from the physics literature  \cite{MMMS}, the second author, C. Leung, and E. Zaslow gave an explicit formulation of an equation on a line bundle that corresponds to the special Lagrangian equation on the mirror \cite{LYZ}. Specifically, let $M$ and $W$ be dual torus fibrations over a base tori $B$. In this semi-flat setting, a connection $A$ on a holomorphic line bundle $L$ over $M$ is dual to a Lagrangian section $\mathcal L$ of the torus fibration $W\rightarrow B$. If $\mathcal L$ is a special Lagrangian section, using the  Fourier-Mukai Transform, Leung-Yau-Zaslow show that $A$ must satisfy the deformed Hermitian-Yang-Mills equation:
 \be
 \label{first formulation}
 {\rm Im}{(\o-F)^n}={\rm tan}\hat\t \,{\rm Re} {(\o-F)^n},
 \ee
where $\o$ is the K\"ahler form on $M$, $F$ is the curvature of the connection $A$, and $\hat\t$ is the phase of the special Lagrangian.

In this paper we undertake a rigorous examination of equation \eqref{first formulation}. Since we are primarily interested in the PDE, we forget about the apparatus of mirror symmetry and only focus on the line bundle side of the above duality. In fact, equation \eqref{first formulation} can be defined on any holomorphic line bundle over a compact K\"ahler manifold, and we choose to work in this general setting. Thus, when the base manifold is not Calabi-Yau, there is no meaningful notion of a mirror special Lagrangian. As a result, we are motivated by mirror symmetry, but never apply it directly. We develop tools for addressing equation \eqref{first formulation} which are in many ways analogous to those used to study special Lagrangian submanifolds, while keeping track of the differences between the two settings. We hope that these results are not only interesting in their own right, but could possibly shed new light on the Lagrangian case.

Here we describe our main results. Let $(X,\o)$ be compact K\"ahler manifold of complex dimension $n$, and let $L$ be a holomorphic line bundle over $X$. Given a metric $h$ on $L$, we define the complex function $\zeta:X\rightarrow\C$ by
\be
\zeta:=\frac{(\o-F)^n}{\o^n}.\nonumber
\ee
The average of this function is the fixed complex number $Z_L:=\int_X\zeta\,\frac {\o^n}{n!}$, which is independent of the choice of metric on $L$. Let $\t$ denote the argument of $\zeta$, and $\hat\t$ the argument of $Z_L$. We show that a metric $h$ solving $\t\equiv\hat\t$ is equivalent to a solution of equation \eqref{first formulation}, and we use this formulation to conclude that solutions of \eqref{first formulation} are always elliptic. Next, we define a functional on the space of metrics:
\be
V(h)=\int_X|\zeta|\,\frac{\o^n}{n!}.\nonumber
\ee
This functional has the property that critical points correspond to solutions of  \eqref{first formulation}, and that all critical points are absolute minima. We use these properties to prove the following uniqueness result:
\begin{thm}
\label{uniqueness}
Let $L$ be a holomorphic line bundle over a compact K\"ahler manifold $X$. Suppose there exists a metric $h$ on $L$ that solves equation \eqref{first formulation}. Then any other solution is a real constant multiple of $h$.  
\end{thm}
To prove this theorem we use the fact that solutions to \eqref{first formulation} behave in a similar fashion to calibrated submanifolds, and that the functional $V(\cdot)$ acts like a volume functional (see \cite{HL}). This is one instance where the analogy with the special Lagrangian equation is quite strong.

We now turn to our existence results. In the case that $X$ is a K\"ahler surface, we show that a solution to \eqref{first formulation} exists if and only if $L$ satisfies a geometric stability condition, which we describe as follows. First, we solve the equation in the trivial case $\hat\t=0$. Then, by possibly looking at $L^{-1}$ instead of $L$, we can assume without loss of generality that $\hat\t>0$. Consider the following K\"ahler form on $X$:
\be
\Omega:={\rm cot}(\hat\t)\omega+iF.\nonumber
\ee 
Note that $\Omega$ depends on our choice of metric $h$, while the K\"ahler class $[\Omega]$ does not. We say the line bundle $L$ is {\it stable} if there exists a metric $h$ on $L$ so that $\Omega>0$. With this definition, we prove the following result:
\begin{thm}
\label{surface}
Let $L$ be a holomorphic line bundle over a K\"ahler surface $X$. Then $L$ admits a solution to \eqref{first formulation} if and only if $L$ is stable.
\end{thm}
We prove the above theorem by transforming equation  \eqref{first formulation} into a complex Monge-Amp\`ere equation, and applying the second author's solution of the Calabi conjecture \cite{Y}. While this result completely settles existence of a solution to \eqref{first formulation} in the surface case, this method does not seem to generalize easily to higher dimensions. When $\t$ satisfies the supercritical phase condition (see section \ref{intro} for a definition), we can extend our stability condition to higher dimensions, and prove it is necessary, although we do not know if this condition is sufficient to guarantee existence.

In order to address existence in higher dimensions, we define a parabolic evolution equation for the metric $h$ on $L$ which is the gradient flow of the functional $V(\cdot)$. Given an initial metric $h_0$ on $L$, we define a flow of metrics $h_t=e^{-\phi(t)}h_0$ by the following equation:
\be
\label{flow1}
\frac{d}{dt}\phi(t)=\t-\hat\t.
\ee
Due to the similarities between the above evolution equation and the Lagrangian mean curvature flow with a potential, we refer to \eqref{flow1} as the {\it line bundle mean curvature flow}. Lagrangian mean curvature flow has been extensively studied, for example see \cite{CLT, H,N, S1, S2, SW,TY, TW,W} and references therein, and certain aspects of Lagrangian mean curvature flow carry over to the evolution equation \eqref{flow1}. However, there are many important differences and difficulties to take into consideration. For example one can define a metric $\eta_{\bar kj}$ on $T^{1,0}(X)$ that is analogous to the induced metric on the Lagrangian submanifold. Unfortunately, this metric is not K\"ahler, thus interchanging derivatives produces torsion terms which complicate many of the maximum principle arguments used for mean curvature flow. At an even more basic level, because our setting does not include an actual Lagrangian submanifold, many of the standard arguments involving decomposing the tangent space into normal and tangential directions can not be applied.

Despite these difficulties we develop some tools for working with the evolution equation \eqref{flow1}, and prove convergence in certain cases. Our first convergence result is as follows

\begin{thm}
\label{bisectional}
Let $L$ be an ample line bundle over a compact K\"ahler manifold $X$ with non-negative orthogonal bisectional curvature. There exists a natural number $k$ so that $L^{\otimes k}$ admits a solution to \eqref{first formulation}. Furthermore, this solution is constructed via a smoothly converging family of metrics along the line bundle mean curvature flow.

\end{thm}
The assumption that $X$ has non-negative orthogonal bisectional curvature comes up in an application of the maximum principle, and we are hopeful that in the future this assumption can be removed. Here we note that the condition of non-negative orthogonal bisectional curvature is slightly more general than that of nonnegative bisectional curvature (see \cite{GZ}). The ampleness assumption is more central to our argument. It corresponds to the condition that the phase $\t$ remains large along the flow, ensuring the curvature $iF$ remains a positive $(1,1)$ form. The positivity of $iF$ implies the operator $h\mapsto\t(h)$ is concave, which plays an important role in the regularity theory of our equation.  In fact, at this point we carry out a complex analogue of the argument of Smoczyk-Wang \cite{SW2}, who prove convergence of the Lagrangian mean curvature flow with convex potential in a flat torus fibration.  The difference is that  in \cite{SW2}, the authors consider a potential with positive Hessian, while our potential is plurisubharmonic.

Because of the important role ampleness plays in the regularity, we do not expect this assumption can be easily removed. It would be interesting to relate ampleness to a notion of stability, as in the complex surface case, and in addition prove that convergence of the flow is dependent upon a stability condition. We note that the second author, along with R. Thomas,  developed a notion of stability for special Lagrangian submanifolds, and conjectured that stability implies convergence of the Lagrangian mean curvature flow \cite{TY}. However, at the moment we can not see how this can be incorporated into our setup. 

We conclude the paper with a final convergence result, proving that the flow converges as long as $\nabla F$, the line bundle analogue of the second fundamental form of a Lagrangian submanifold, stays bounded in $C^0$. 
 
\begin{thm}
\label{higherbounds}
Let $L$ be a holomorphic line bundle over a compact K\"ahler manifold $X$, and let $h_t$ be a path of metrics on $L$ solving \eqref{flow} on the time interval $[0,T)$, with $T\leq\infty$. Assume that the following quantity is uniformly bounded in time
\be
\label{assump1}
|\nabla F|^2_g\leq C_1.
\ee
Then there exists constants $C_k$, depending only on $C_1$ and initial data, bounding all higher order derivatives of $F$ along the flow:
\be
\label{goal1}
|\nabla^k F|^2_g\leq C_k.
\ee 
If $T$ is finite, then $h_t$ converges in $C^\infty$ to a limiting metric $h_T$, and the flow can be continued. If $T=\infty$, then there exists a subsequence of times along the flow which converge to a smooth solution $h_\infty$ of \eqref{first formulation}.

\end{thm}
As a consequence of this theorem we see that either $|\nabla F|_g$ blows up at a finite time along the flow, creating a singularity, or it is bounded, and the flow exists for all time. Again, because the term $\nabla F$ is the line bundle equivalent of the second fundamental form, the above result is a line bundle analogue of Theorem 1.2 from \cite{S1}.

The paper is organized as follows. In Section \ref{intro} we give our main definitions, and explicitly write down several forms of the deformed Hermitian-Yang-Mills equation \eqref{first formulation} that will be used in later sections. In Section \ref{volume} we introduce the volume functional, prove that critical points must satisfy \eqref{first formulation}, and then prove that solutions are unique. Section \ref{stability} contains our complete existence results in the case of K\"ahler surfaces. Finally, in Section \ref{convergence}, we intoduce the line bundle mean curvature flow and prove our two convergence results.

\vspace{\baselineskip}
\noindent
{\bf Acknowledgements} We would like to express our gratitude to Murad Alim, Tristan C. Collins, Siu-Cheong Lau, Duong H. Phong, Valentino Tosatti and Mu-Tao Wang for many helpful discussions and comments.

\section{The deformed Hermitian-Yang-Mills equation}
\label{intro}
We begin this section with a few preliminary definitions, and introduce our notation and conventions. Let $(X,\omega)$ be a compact K\"ahler manifold. The K\"ahler form $\omega$ is a closed $(1,1)$ form, locally expressed as
\be
\o=\frac i2g_{\bar kj}dz^j\wedge d\bar z^k,\nonumber
\ee
where $g_{\bar kj}$ is a Hermitian metric on the holomorphic tangent bundle $T^{1,0}(X)$. The top dimensional form $\frac{\o^n}{n!}$ defines a natural volume form on $X$, and we normalize $\o$ so that $X$ has volume one. Let $\Lambda^{p,q}(X)$ denote the bundle of $(p,q)$ forms, and note that the metric $g_{\bar kj}$ induces a metric on each of these bundles as well.  We define the Chern connection $\nabla$  on $T^{1,0}(X)$ to be the unique connection that preserves the metric $g_{\bar kj}$ and defines the holomorphic structure. Given $\nabla$, we extend this connection to the associated bundles $\Lambda^{p,q}(X)$ in the standard fashion.

Fix a holomorphic line bundle $L$ over $X$. Given a metric $h$ on $L$, the curvature two form can be expressed in local coordinates as
 \be
 F=\frac12 F_{\bar kj} dz^j\wedge d\bar z^k:=-\frac12\pl_j\pl_{\bar k}{\rm log}(h) dz^j\wedge d\bar z^k.\nonumber
 \ee
Using $F$, we introduce another Hermitian metric $\eta_{\bar kj}$ on $T^{1,0}(X)$ that plays an important role in this paper, defined by 
 \be
\eta_{\bar kj}=g_{\bar kj}+F_{\bar k\ell}g^{\ell \bar m}F_{\bar mj}.\nonumber
\ee
This metric is the line bundle analogue of the induced metric on a Lagrangian submanifold. Note that $\eta_{\bar kj}$ is not a K\"ahler metric. Nevertheless, we use $\eta_{\bar kj}$ to define the following Laplacian on $C^\infty(X)$:
\be
\Delta_\eta(f)=\eta^{j\bar k}\pl_j\pl_{\bar k}(f).\nonumber
\ee
We remark this operator is defined using the ``analyst convention," and is elliptic as long as the curvature $F$ is bounded.

Associated to the line bundle $L$ we have the following two invariants.
\begin{defn}
Given a holomorphic line bundle $L$ over $X$, we define following fixed complex number
\be
Z_L:=\int_X \frac{(\o-F)^n}{n!},\nonumber
\ee
as well as the following angle:
\be
\hat\t:={\rm arg}(Z_L).\nonumber
\ee
\end{defn}
These two invariants are independent of the choice of metric $h$, since for $1\leq k\leq n$, the form $F^k$ differs from the $k$-th degree term of the Chern character of $L$ by a constant. We view arg$(\cdot)$ as a function from $\C$ to $\R$, as opposed to the unit circle, by specifying the argument to be zero precisely when the curvature $F$ vanishes.

The deformed Hermitian-Yang-Mills equation seeks a metric $h$ on $L$ so that the function
 \be
\zeta:=\frac{(\o-F)^n}{\o^n}:X\longrightarrow \C\nonumber
 \ee
 has constant argument. It is not hard to see that if a solution exists, the constant argument must in fact be equal to $\hat\t$. Thus we are looking for a solution to the equation
 \be
 \label{DHYMn}
 {\rm Im} \frac{(\o-F)^n}{\o^n}={\rm tan}(\hat\t)\,{\rm Re} \frac{(\o-F)^n}{\o^n}.
 \ee
As stated in the introduction, in the case where $X$ is a Calabi-Yau manifold, the above equation is the line bundle analogue of the equation for special Lagrangian section of a torus fibration. Again, in this paper we consider the general case where $X$ is a compact K\"ahler manifold, so we never appeal to mirror symmetry directly, and only use it for motivation.

Now, in its current form, it is not obvious that equation \eqref{DHYMn} is elliptic. We therefore derive two equivalent formulations of \eqref{DHYMn} in which ellipticity is more apparent. First, we define an endomorphism $K$ of $T^{1,0}(X)$ via the contraction
 \be
 K:=g^{j\bar k}F_{\bar k\ell} \, \frac{\pl}{\pl z^j}\otimes dz^\ell.\nonumber
 \ee 
 Now, the top dimensional form  $(\o-F)^n$ is written locally as
 \bea
\frac{(\o-F)^n}{n!}&=& {\rm det}\left(g_{\bar kj}+iF_{\bar kj}\right)\left(\frac i2\right)^ndz^1\wedge d\bar z^1\wedge\cdots\wedge dz^n\wedge d\bar z^n\nonumber\\
&=&{\rm det}(g^{j\bar k}){\rm det}\left(g_{\bar kj}+iF_{\bar kj}\right)\frac{\o^n}{n!}\nonumber\\
&=&{\rm det}(I+iK)\frac{\o^n}{n!}.\nonumber
 \eea
 Thus we see the complex function $\zeta$ is equal to
\be
\label{zeta}
\zeta={\rm det}(I+iK),
\ee
and its modulus can be expressed as
 \be
 |\zeta|=\sqrt{\bar\zeta\zeta}=\sqrt{{\rm det}(I-iK){\rm det}(I+iK)}=\sqrt{{\rm det}(I+K^2)}.\nonumber
 \ee
The above formula implies $|\zeta|\geq 1$, since the matrix $K^2$ is positive definite. As a result we know that $\zeta/|\zeta|$ is a unit complex vector, whose argument is always defined and given by 
\be
\label{theta}
\t=-i{\rm log}\frac{{\rm det}(I+iK)}{\sqrt{{\rm det}(I+K^2)}}.
\ee
Again, we specify a branch cut for the logarithm by having $\t=0$ correspond to the case that $K=0$. In this way we can view $\t$ as a real function on $X$.

Now, at a point $p$ we can choose coordinates so $g_{\bar kj}=\d_{kj}$ and $F_{\bar kj}=\lambda_j\delta_{kj}$, and we refer to this choice of coordinates as {\it normal coordinates}. In normal coordinates the endomorphism $K$ is diagonal at the point $p$ with eigenvalues $\{\lambda_j\}$ for $1\leq j\leq n$, and the above formula for $\t$ becomes 
\bea
\label{arctan}
\t&=&-i{\rm log}\frac{\Pi_j(1+\lambda_j)}{(\Pi_j(1+\lambda_j)\Pi_i(1-\lambda_j))^{\frac12}}=-\frac i2{\rm log}\frac{\Pi_j(1+\lambda_j)}{\Pi_i(1-\lambda_j)}\nonumber\\
&=&\sum_j\left(\frac i2{\rm log}(1-\lambda_j)-\frac i2{\rm log}(1+\lambda_j)\right)=\sum_j{\rm arctan}(\lambda_j).
\eea
Once again we seek a metric $h$ on $L$ so that $\t\equiv\hat\t$. In this formulation it is clear that our equation is analogus to the special Lagrangian equation (for instance see the work of  Chen-Warren-Yuan \cite{CWY} and Wang-Yuan \cite{ WY1, WY2}).

In the special Lagrangian case certain bounds on the phase affect the behavior of the equation, and in particular $C^2$ bounds for the potential are easier to attain for large phase (see \cite{CWY,WY1,WY2} and references therein). This ends up being the case in our setting as well, and the following definitions will play an important role in the convergence results to follow.
 \begin{defn}
We say $\t$ satisfies the supercritical phase condition if 
 \be
|\t|> (n-2)\frac\pi2.\nonumber
 \ee
Furthermore, $\t$ satisfies the hypercritical phase condition if in addition 
\be
|\t|> (n-1)\frac\pi2.\nonumber
\ee
 \end{defn}
 We end this section with the following definition
 \begin{defn}
 The mean curvature 1 form of a metric $h$ on $L$ is defined by
 \be
 \label{mc}
 H:= d\t.
 \ee
 \end{defn}
Because a solution to \eqref{DHYMn} is given by a metric $h$ for which $\zeta$ has constant argument, it is clear such solutions have vanishing mean curvature.

\section{The volume functional and uniqueness}
\label{volume}

We begin this section by proving solutions to \eqref{DHYMn} minimize a volume type functional. The argument given here is analogous to the argument showing special Lagrangian submanifolds are calibrated submanifolds and thus minimal. We next show that all critical points of this functional solve \eqref{DHYMn}, and use this fact to prove Theorem \ref{uniqueness}. Recall the modulus of the complex function $\zeta$ is given by:
\be
|\zeta|={\sqrt{{\rm det}(I+K^2)}}.\nonumber
\ee
This leads to the following definition.
\begin{defn} The volume functional $V(h)$ of a metric is defined as follows:
\be
V(h):=\int_X|\zeta|\,\frac{\o^n}{n!}.\nonumber
\ee
\end{defn}
The above functional is clearly positive. We now show that any constant mean curvature metric is a global minimizer of $V(h)$.
\begin{prop}
\label{globalmin}
Let $h$ be a metric on $L$ with constant mean curvature. Then $h$ must be a global minimizer of $V(\cdot)$. Furthermore, at the minimizer $V(h)=|Z_L|$, which implies $V(\cdot)$ has a positive lower bound depending only on the Chern numbers of $L$. 
\end{prop}
\begin{proof}
First, we consider the real top dimensional form given by
\be
\Psi={\rm Re}\left(e^{-i\hat\t}\frac{(\o-F)^n}{n!}\right).\nonumber
\ee
Note that the integral of $\Psi$ is independent of metric. This observation is important, as $\Psi$ will play the analogue of a calibration form.  Because $\Psi$ is a real top dimensional form, it can be expressed as $\Psi=\kappa \,|\zeta|\,\frac{\o^n}{n!}$, where $\kappa$ is given by
\be
\label{kappa}
\kappa={\rm Re}\left(e^{-i\hat\t}\frac{\zeta}{|\zeta|}\right).
\ee
Since both $e^{-i\hat\t}$ and ${\zeta}/{|\zeta|}$ are unit length, it is clear that $\kappa\leq 1$. Now, suppose $h$ has constant mean curvature. Then by equation \eqref{mc} we see that $\t=\t(h)$ is locally constant (and therefore constant since we only consider $X$ with one connected component), which implies $\t\equiv\hat\t$. In this case we see that $\kappa$ is in fact equal to one, since
\be
\kappa={\rm Re}\left(e^{-i\hat\t}e^{i\t}\right)={\rm Re}\left(e^{-i\hat\t}e^{i\hat\t}\right)=1.\nonumber
\ee
Thus, for any other metric $h'$ we have
\be
V(h)=\int_X|\zeta(h)|\,\frac{\o^n}{n!}=\int_X\Psi(h)=\int_X\Psi(h')=\int_X\kappa\, |\zeta(h')|\,\frac{\o^n}{n!}\leq V(h'),\nonumber
\ee
where the third equality follows from the fact that the integral of $\Psi(h)$ is independent of metric, and the inequality follows since $\kappa\leq1$. This proves $h$ is a global minimizer for $V(\cdot)$. Again, because the integral of $\Psi(h)$ is independent of metric, the global minimum is given by
\bea
\int_X\Psi=\int_X{\rm Re}\left(e^{-i\hat\t}\frac{(\o-F)^n}{n!}\right)&=&{\rm Re}\left(e^{-i\hat\t}\int_X\frac{(\o-F)^n}{n!}\right)\nonumber\\
&=&{\rm Re}\left(e^{-i\hat\t}Z_L\right)=|Z_L|.\nonumber
\eea
This gives a positive lower bound for $V(\cdot)$. 
\end{proof}

Next we show that at any critical point of $V$, the metric $h$ must have constant mean curvature, which implies that all critical points of $V(h)$ are in fact global minimizers of $V(h)$. Before we can prove this, we need the following lemma.
\begin{lem}
The variation of $\t$ is given by
\be
\delta\t={\rm Tr}((I+K^2)^{-1}\delta K).\nonumber
\ee
\end{lem}
\begin{proof}
We begin by computing the variation directly using equation \eqref{theta}:
\be
\d\t=-i{\rm Tr}((I+iK)^{-1}i\delta K)+\frac i2{\rm Tr}((I+K^2)^{-1}(\delta K K+K\delta K)).\nonumber
\ee
Working in normal coordinates the endomorphism $K$ is diagonal at $p$, which we denote by $K=diag(\lambda_1,...,\lambda_n)$. One sees that the matrix $I+K^2$ is diagonal as well, from which it follows that $K$ and $(I+K^2)^{-1}$ commute. Furthermore, in these coordinates it is easy to see the identity
\be
\label{ident}
(I+iK)^{-1}=(I+K^2)^{-1}-iK(I+K^2)^{-1}.
\ee
Simply note that
\be
(I+iK)^{-1}=diag\left(\frac1{1+i\lambda_1},...,\frac1{1+i\lambda_n}\right)=diag\left(\frac{1-i\lambda_1}{1+\lambda_1^2},...,\frac{1-i\lambda_n}{1+\lambda_n^2}\right).\nonumber
\ee
Separating out the real and imaginary parts from each eigenvalue proves \eqref{ident}. Putting together everything so far, we see
\be
\d\t={\rm Tr}\left((I+K^2)^{-1}\delta K-iK(I+K^2)^{-1}\delta K+ i(I+K^2)^{-1}\delta K K\right).\nonumber
\ee
The imaginary terms cancel, completing the lemma. 
\end{proof}

We can now write down a local coordinate version of the mean curvature one form $H$. Using the metric $\eta_{\bar kj}$, the endomorphism $I+K^2$ can be expressed locally as 
\be
I+K^2=g^{p\bar q}\eta_{\bar q\ell} \frac{\pl}{\pl z^p}\otimes dz^\ell.\nonumber
\ee
The mean curvature one form $H$ can now be written as $H_j dz^j$, where $H_j$ is given by
\be
\label{mcloc}
H_j=\pl_j\t={\rm Tr}((I+K^2)^{-1}\nabla_jK)=\eta^{p\bar q}g_{\bar q\ell}\nabla_j(g^{\ell\bar m}F_{\bar mp})=\eta^{p\bar q}\nabla_jF_{\bar qp}.
\ee
The last equality follows because $\nabla$ passes through the metric $g_{\bar kj}$. Because the metric $\eta_{\bar kj}$ is analogous to the induced metric on a special Lagrangian submanifold, from the above equation we see that $\nabla F$ is the line bundle analogue of the second fundamental form.

The above lemma also lets us compute the time derivative of $\t$ along a path of metrics. Let $h(t)=e^{-\phi(t)}h_0$ be a smooth path of metrics, where $h_0$ is an initial metric and $\phi(t)$ is a real function. Then as before one computes
\be
\label{evolvet}
\dot \t={\rm Tr}((I+K^2)^{-1}\dot K)=\eta^{p\bar q}g_{\bar q\ell}\frac{d}{dt}(g^{\ell\bar m}F_{\bar mp})=\eta^{p\bar q}\pl_p\pl_{\bar q}\dot\phi=\Delta_{\eta}\dot\phi.
\ee
The above formula will be utilized many times in the analysis to follow. We are now ready to compute the variation of the volume functional $V(\cdot)$.

\begin{prop}
\label{criticalpoint}
Let $h(t)=e^{-\phi(t)}h_0$ be any smooth path of metrics on $L$. Then at a critical point $h$ of $V$ we have
\be
0=\frac d{dt}V(h)=-\int_X\langle H,d\dot\phi\rangle_\eta\,|\zeta|\,\frac{\o^n}{n!}.\nonumber
\ee
In particular at any critical point of $V$ the metric $h$ must have vanishing mean curvature. 
\end{prop}

\begin{proof}
 For notational simplicity we denote the function $|\zeta|$ by $v$. To prove the proposition, we first compute the time derivative of $v$:
\be
\label{vold2}
\dot v=\frac{{\rm det}(I+K^2){\rm Tr}((I+K^2)^{-1}2K\dot K)}{2\sqrt {{\rm det}(I+K^2)}}={\rm Tr}((I+K^2)^{-1}K\dot K)\,v.
\ee
Writing the above equation in coordinates we have
\be
\dot v=\eta^{j\bar k}F_{\bar k\ell}g^{\ell\bar q}\nabla_j\nabla_{\bar q} \dot\phi\, v,\nonumber
\ee
which is equal to 
\be
\dot v=\nabla_j\left(\eta^{j\bar k}F_{\bar k\ell}g^{\ell\bar q}\nabla_{\bar q}\dot\phi\, v \right)-\nabla_j\left(\eta^{j\bar k}F_{\bar k\ell} \, v\right)g^{\ell\bar q}\nabla_{\bar q}\dot\phi.\nonumber
\ee
We now concentrate on the right most term. Applying the product rule gives
\bea
-\nabla_j\left(\eta^{j\bar k}F_{\bar k\ell} \, v\right)g^{\ell\bar q}\nabla_{\bar q}\dot\phi&=&-\nabla_j\left(\eta^{j\bar k} \,v\right)F_{\bar k\ell}g^{\ell\bar q}\nabla_{\bar q}\dot\phi-\eta^{j\bar k}\nabla_jF_{\bar k\ell} g^{\ell\bar q}\nabla_{\bar q}\dot\phi\, v\nonumber\\
&=&-\nabla_j\left(\eta^{j\bar k} \,v\right)F_{\bar k\ell}g^{\ell\bar q}\nabla_{\bar q}\dot\phi-H_\ell g^{\ell\bar q}\nabla_{\bar q}\dot\phi\, v,\nonumber
\eea
where we used equation \eqref{mcloc} on the second term on the right. In order to understand the first term on the right, we apply the product rule as follows:
\be
\label{step}
 \nabla_{j}(\eta^{j\bar k}\, v)=-\eta^{j\bar s}\nabla_j\eta_{\bar sr}\eta^{r\bar k}\, v+\eta^{j\bar k}\eta^{r\bar s} F_{\bar su}g^{u\bar q}\nabla_jF_{\bar qr}\, v.\nonumber
\ee
Expanding out the derivative of $\eta_{\bar sr}$, the above equation becomes
\be
-\eta^{j\bar s}\nabla_jF_{\bar su}g^{u\bar q}F_{\bar qr}\eta^{r\bar k}\, v-\eta^{j\bar s}F_{\bar su}g^{u\bar q}\nabla_jF_{\bar qr}\eta^{r\bar k}\, v+\eta^{j\bar k}\eta^{r\bar s} F_{\bar su}g^{u\bar q}\nabla_jF_{\bar qr}\, v.\nonumber
\ee
Note that in normal coordinates the metric $\eta_{\bar kj}$ is also diagonal, and given by $\eta_{\bar kj}=\d_{kj}(1+\lambda_j^2)$. We see the second two terms on the right cancel because they are equivalent to the sum
\be
\sum_{kp}\left( \eta^{k\bar k}F_{\bar kk}\nabla_kF_{\bar kp}\eta^{p\bar p}- \eta^{p\bar p}\eta^{k\bar k} F_{\bar kk}\nabla_pF_{\bar kk}\right)\, v,\nonumber
\ee
which vanishes by the second Bianchi identity. Putting everything together so far, we have
\be
\dot v=\nabla_j\left(\eta^{j\bar k}F_{\bar k\ell}g^{\ell\bar q}\nabla_{\bar q}\dot\phi\, v \right)-H_\ell g^{\ell\bar q}\nabla_{\bar q}\dot\phi\, v+\eta^{j\bar s}\nabla_jF_{\bar su}g^{u\bar q}F_{\bar qr}\eta^{r\bar k}F_{\bar k\ell}g^{\ell\bar q}\nabla_{\bar q}\dot\phi\, v.\nonumber
\ee
We again apply \eqref{mcloc} to see $\eta^{j\bar s}\nabla_jF_{\bar su}=H_u$. Now, making use of the identity 
\be
I=(I+K^2)^{-1}+(I+K^2)^{-1}K^2,\nonumber
\ee
one can prove
\be
\label{ginverse}
g^{\ell\bar q}=\eta^{\ell\bar q}+\eta^{\ell\bar k}F_{\bar kr}g^{r\bar s}F_{\bar sj}g^{j\bar q}.
\ee
This fact, along with the observation that $K$ commutes with $(I+K^2)^{-1}$, gives
\be
-H_\ell g^{\ell\bar q}\nabla_{\bar q}\dot\phi\, v+H_ug^{u\bar q}F_{\bar qr}\eta^{r\bar k}F_{\bar k\ell}g^{\ell\bar q}\nabla_{\bar q}\dot\phi\, v=-\eta^{\ell\bar q}H_\ell \nabla_{\bar q}\dot\phi\, v.\nonumber
\ee
Thus we can conclude
\be
\label{vevolve}
\dot v=-\eta^{\ell\bar q}H_\ell \nabla_{\bar q}\dot\phi\, v+\nabla_j\left(\eta^{j\bar k}F_{\bar k\ell}g^{\ell\bar q}\nabla_{\bar q}\dot\phi\, v \right).
\ee
Note the second term on the right is the divergence of a vector field, and therefore vanishes when integrating over $X$. This gives
\be
\int_X\dot v\,\frac{\o^n}{n!}=-\int_X\langle H,d\dot\phi\rangle_\eta\, v\,\frac{\o^n}{n!}.\nonumber
\ee
The term on the left is none other than $\frac{d}{dt} V(h)$. This completes the proof of the proposition.

\end{proof}

We can now prove our main uniqueness result.

\begin{proof}[Proof of Theorem \ref{uniqueness}]
Consider two metrics, $h$ and $h'$, and assume they both have constant mean curvature. Then by Proposition \ref{globalmin} we know both $h$ and $h'$ are global minimizers of $V$. Recall that at a global minimum the volume $V$ is specified, thus $V(h)=|Z_L|=V(h')$. Now, for any two metrics there exists a smooth function $\phi$ such that $h=e^{\phi}h'$, and we can define a path of metrics $h_t:=e^{\phi t}h'$, $t\in[0,1]$, satisfying $h_0=h'$ and $h_1=h$. Suppose that not every metric along the path is a global minimizer. Then there exists  a time $T\in[0,1]$ and a metric $h_T$ along the path where $V(h_T)$ achieves its maximum. At this point $\frac d{dt}V(h_T)=0$, which implies $h_T$ has constant mean curvature by Proposition \ref{criticalpoint}, and thus $V(h_T)=|Z_L|$ by Proposition \ref{globalmin}. We conclude that $V(h_t)$ must be constant along the path $h_t$.

Next we prove $h_t$ is a metric of constant mean curvature for each $t$.  Because $V(h_t)=|Z_L|$ for each time time $t$, we have
\be
\int_X|\zeta(h_t)|\,\frac{\omega^n}{n!}=\int_X\Psi=\int_X\kappa(h_t)\,|\zeta(h_t)|\,\frac{\o^n}{n!}.\nonumber
\ee
Furthermore, because $\kappa(h_t)\leq1$, to achieve the equality above we must have $\kappa(h_t)\equiv1$. Let $\t_t$ be the argument of $\zeta(h_t)$. Applying the definition of $\kappa$ \eqref{kappa}, we have
\be
1\equiv{\rm Re}\left(e^{-i\hat\t}\frac{\zeta}{|\zeta|}\right)={\rm Re}\left(e^{-i(\hat\t-\t_t)}\right)={\rm cos}\,(\hat\t-\t_t).\nonumber
\ee
This implies $\t_t$ is fixed and as a result $h_t$ has constant mean curvature.

We can now complete the proof of uniqueness. We have shown that $h_t$ has constant mean curvature for all $t$ along the path, which implies $\t_t\equiv \hat\t$ for all $t$. Equation \eqref{evolvet} gives
\be
0=\frac{d}{dt}\t(h_t)=\Delta_{\eta_t}\frac{d}{dt}(\phi t)=\Delta_{\eta_t}\phi.\nonumber
\ee
By the maximum principle we conclude that $\phi$ is constant.
\end{proof}

\section{Stability and existence on surfaces}
\label{stability}

In this section we define a notion of stability for line bundles in the case that $X$ is a K\"ahler surface. We then prove that a solution to \eqref{DHYMn} exists if and only if $L$ is stable. At the end of the section we extend our notion of stability to higher dimensions in the case of supercritical phase, and prove our stability is necessary.

For the remained of this section we assume, without loss of generality, that $\hat\t\geq0$.  If $\hat\t<0$, one can instead work on $L^{-1}$ to change the sign of $\hat\t$, since $Z_{L^{-1}}=\bar Z_L$. We begin by explicitly writing down the complex number $Z_L$, using the fact that $X$ has complex dimension two.
\bea
Z_L=\int_X \frac{(\o-F)^2}{2}&=&\int_X\frac{\o^2}2-\int_X\frac{(iF)^2}2+i\int_X(iF)\wedge\omega.\nonumber
\eea
For notational simplicity, we define the fixed constants
\be
\qquad a_1=\int_X(iF)\wedge\omega\qquad{\rm and}\qquad a_2=\int_X\frac{(iF)^2}2,\nonumber
\ee
which differ from the usual numbers defined by  the Chern character by factors of $2\pi$. We now express $Z_L$ as follows:
\be
Z_L=1-a_2+ia_1.\nonumber
\ee
This leads to the simple observation that
\be
\label{cot}
{\rm cot}(\hat\t)=\frac{1-a_2}{a_1}.
\ee

Note that ${\rm cot}(\hat\t)$ is undefined when $\hat\t=k \pi$ for $k\in \Z$. However, in the case that $\hat\t=0$, the line bundle $L$ has degree zero, and thus there exists a metric $h$ on $L$ so that $F\wedge\omega=0$. In this case it is clear that Im$(\zeta)\equiv0$, and as a result $h$ has constant mean curvature. Thus we can assume $\hat\t\neq 0$. Furthermore, by the arctan definition of $\t$, in normal coordinates we have
\be
\t={\rm arctan}(\lambda_1)+{\rm arctan}(\lambda_2).\nonumber
\ee
Since the image of ${\rm arctan}(\cdot)$ lies in the open interval $(-\frac\pi2,\frac\pi2)$, we see that $\t<\pi$, and as a result $\hat\t\in(0,\pi)$. Thus we can assume \eqref{cot} is well defined.
%First I have to confirm, if $a_1=0$, then $deg(L)=0$ and this implies $a_2=0$ as well. Also, as in the previous sections, we assume without loss of generality that $a_1\geq0$. Thus we are looking at $\hat\t\in[0,\pi)$. Furthermore, 

Consider the following closed, real $(1,1)$ form:
\be
\Omega:={\rm cot}(\hat\t)\omega+iF.\nonumber
\ee 
Note that the form $\Omega$ depends on a choice of metric $h$ on $L$, however the cohomology class $[\Omega]$ does not. 
\begin{defn}
We say the line bundle $L$ is stable if there exists a metric $h$ on $L$ such that $\Omega>0$.
\end{defn}
By the result of Demailly-Paun \cite{DP} (see also \cite{CT}), the above stability is equivalent to the following geometric condition.
\begin{defn}
We say the line bundle $L$ is stable if for all irreducible analytic cycles $Y\subset X$ of complex dimension one, the integral of $[\Omega]$ is positive:
\be
\int_Y\Omega>0.\nonumber
\ee
\end{defn}
With stability now defined, we can prove the main result of this section

\begin{proof}[Proof of Theorem \ref{surface}]
We first prove that the existence of a constant mean curvature metric implies $L$ is stable. Let $h$ be a solution to $\eqref{DHYMn}$, and let $F$ be the curvature of $h$. At an arbitrary point $p\in X$, choose coordinates so that the endomorphism $K$ is diagonal with eigenvalues $\lambda_1$ and $\lambda_2$. By equation \eqref{arctan}, we know
\be
{\rm arctan}(\lambda_1)+{\rm arctan}(\lambda_2)=\hat\t.
\ee
Again, because the image of ${\rm arctan}(\cdot)$ lies in the open interval $(-\frac\pi2,\frac\pi2)$, we know for $j=1,2$ the following inequality holds:
\be
{\rm arctan}(\lambda_j)\geq\hat\t-\frac{\pi}2.\nonumber
\ee
Now, we assumed that $\hat\t\in(0,\pi)$, which implies $\hat\t-\frac{\pi}2\in(-\frac\pi2,\frac\pi2)$. Because tan$(\cdot)$ is an increasing function on this interval, we can apply it to both sides of the above inequality to conclude
\be
\lambda_j\geq{\rm tan}(\hat\t-\frac{\pi}2)=-{\rm cot}(\hat\t).\nonumber
\ee
From the definition of $K$ it is clear the above inequality implies $\Omega>0$. 

We now prove stability implies existence of a solution to \eqref{DHYMn}. Let $h_0$ be a metric such that $\Omega_0:=\Omega(h_0)>0$, and let $F_0$ denote the curvature of $h_0$. We construct a solution to $\eqref{DHYMn}$ by finding a smooth real function $\phi$ on $X$ so that
\be
\label{firsteq}
(\Omega_0+\frac{i}{2}\pl\bar\pl\phi)^2=(1+{\rm cot}^2(\hat\t))\,\o^2.
\ee
To solve this equation, we first check that both sides define the same volume. To see this, we simply inegrate
\bea
\int_X\frac{\Omega_0^2}2&=&\int_X\frac{({\rm cot}(\hat\t)\o+iF_0)^2}2\nonumber\\
&=&\int_X{\rm cot}^2(\hat\t)\frac{\o^2}2+\int_X\frac{(iF_0)^2}2+{\rm cot}(\hat\t)\int_X(iF_0)\wedge\omega\nonumber\\
&=&{\rm cot}^2(\hat\t)+a_2+{\rm cot}(\hat\t)a_1=1+{\rm cot}^2(\hat\t),\nonumber
\eea
where the last equality follows from \eqref{cot}. This observation, along with the assumption that $\omega^2/2$ integrates to one, demonstrates that both sides of equation \eqref{firsteq} define the same volume.

Consider the following smooth function on $X$:
\be
G:={\rm log}\left(\frac{(1+{\rm cot}^2(\hat\t)){\rm det}(g_{\bar kj})}{{\rm det}({\rm cot}(\hat\t)g_{\bar kj}+F^0_{\bar kj})}\right)\nonumber.
\ee
We now rewrite equation \eqref{firsteq} as
\be
(\Omega_0+\frac{i}{2}\pl\bar\pl\phi)^2=e^G\Omega_0^2. \nonumber
\ee
By the second author's proof of the Calabi conjecture \cite{Y} we know there exists a smooth solution $\phi$ satisfying
\be
\Omega:=\Omega_0+\frac{i}2\pl\bar\pl \phi>0\qquad{\rm and}\qquad \sup_X\phi=0,\nonumber
\ee
such that $\phi$ solves \eqref{firsteq}. Given this solution $\phi$, we can define the metric $h=e^{-\phi}h_0$ on $L$. Then \eqref{firsteq} gives
\bea
(1+{\rm cot}^2(\hat\t))\,\o^2&=&({\rm cot}(\hat\t)\omega+iF_0+\frac{i}{2}\pl\bar\pl\phi)^2\nonumber\\
&=&({\rm cot}(\hat\t)\omega+iF)^2\nonumber\\
&=&{\rm cot}^2(\hat\t)\,\o^2+(iF)^2+2{\rm cot}(\hat\t)(iF)\wedge\o.\nonumber
\eea
Subtracting ${\rm cot}^2(\hat\t)\,\o^2$ from both sides, we see the above equation is equivalent to
\be
(iF)\wedge\o={\rm tan}(\hat\t)\left(\frac{\o^2}{2}-\frac{(iF)^2}{2}\right),\nonumber
\ee
which is none other than equation \eqref{DHYMn}. 

\end{proof}

Unfortunately, the above method does not easily generalize to higher dimensions. If $n\geq 3$, one can not rewrite equation \eqref{DHYMn} as a Monge-Amp\`ere equation, but rather a more complicated equation containing linear combinations of elementary symmetric polynomials of the eigenvalues of $K$. While these types of equations have been studied and solved in certain cases \cite{FLM,P1,P2} (see \cite{Kr3} for the real case), it is not known how to construct a solution in general, especially when some of the coefficients from the linear combination are negative. In the future, this may be a reasonable approach to constructing a solution to \eqref{DHYMn}. However, in the following section we propose a different approach that is more closely related to our geometric setup.

We conclude this section by extending our notion of stability to higher dimensions. Again, without loss of generality, assume $\hat\t\geq 0$. Furthermore, we assume $\hat\t$ satisfies the supercritical phase condition. Then, as before we can define the K\"ahler form
\be
\Omega:=-{\rm tan}(\hat\t-(n-1)\frac\pi2)\o+iF,\nonumber
\ee
and define $L$ to be stable if and only if there exists a metric $h$ on $L$ so that $\Omega>0$.

The proof that this condition is necessary follows the same argument as the one given in the proof of Theorem \ref{surface}. Assume we have a solution to the equation
\be
\hat\t=\sum_j{\rm arctan}(\lambda_j).\nonumber
\ee
Again, because the image of ${\rm arctan}(\cdot)$ lies in the open interval $(-\frac\pi2,\frac\pi2)$, we conclude:
\be
{\rm arctan}(\lambda_j)\geq\hat\t-(n-1)\frac{\pi}2\nonumber
\ee
for any $1\leq j\leq n$. By assumption  $\hat\t\in((n-2)\frac\pi2, \frac{n\pi}2)$, which implies $\hat\t-(n-1)\frac{\pi}2\in(-\frac\pi2,\frac\pi2)$. Because tan$(\cdot)$ is an increasing function on this interval, we can apply it to both sides of the above inequality to conclude the stability condition is necessary in this case. It would be interesting to know if this notion of stability is sufficient to guarantee existence, or if it needs to be strengthened.

 \section{The parabolic equation and convergence results}
 \label{convergence}

As stated in the previous section, using a Monge-Ampere type equation to solve \eqref{DHYMn} becomes very difficult in higher dimensions. We instead propose an alternate approach for solving our equation, which is to deform the metric $h$ by the gradient flow of the functional $V(\cdot)$. We call this flow the line bundle mean curvature flow.

The flow is defined as follows. Given initial metric $h_0$, consider a path of metrics $h(t)=e^{-\phi(t)}h_0$, starting at $\phi(0)=0$, defined by the evolution equation 
 \be
 \label{flow}
 \dot\phi=\t-\hat\t=-i{\rm log}\frac{{\rm det}(I+iK)}{\sqrt{{\rm det}(I+K^2)}}-\hat\t.
 \ee
 Taking another derivative in time, by equation \eqref{evolvet} we see right away that
 \be
 \label{thetaflow}
 \ddot\phi=\dot\t=\Delta_\eta\dot\phi.
 \ee
 Thus the flow is parabolic, and exists for a shot time $t\in[0,\epsilon_0)$. An application of equation \eqref{mcloc} gives that $d\phi$ evolves in the direction of the mean curvature one form:
 \be
 \label{evolve1}
d\dot\phi=d\t=H.
 \ee
Now, in the proof of Proposition \ref{criticalpoint}, we computed the variation of the volume functional $V(\cdot)$. Plugging in the above equation gives:
 \be
 \frac{d}{dt}V(h)=-\int_X\langle H,d\dot\phi\rangle_\eta |\zeta|\,\frac{\o^n}{n!}=-\int_X|H|^2_\eta\,|\zeta|\,\frac{\o^n}{n!}\leq 0.\nonumber
 \ee
 Therefore, \eqref{flow} is the gradient flow of the functional $V(h)$.

Although our flow is relatively easy to define, the nonlinearities associated with \eqref{flow} are quite formidable, and higher order estimates of the potential do not follow easily. In fact, in the general case, there is no way to rule out a finite time singularity. While this behavior may be of independent interest due to the relationship with Lagrangian mean curvature flow, it exposes some of the difficulties associated with constructing metrics of constant mean curvature.

 In the rest of this section, we prove convergence in two situations. First, we show that if $X$ has non-negative orthogonal bisectional curvature, and $\t$ satisfies the hypercritical phase condition, then the flow converges to a constant mean curvature metric along a subsequence of times. Second, we show that if $|\nabla F|^2$ is bounded along the flow, then again the flow converges along a subsequence of times. The term $\nabla F$ is the line bundle equivalent of the second fundamental form. Thus, just as in the case of the mean curvature flow, the second fundamental form $\nabla F$ must blow up at any finite time singularity of \eqref{flow}.

Let $L$ be a positive line bundle over $X$, and let $h_0$ be a metric on $L$ so that $iF_0>0$. Recall that the angle $\t(h_0)$ is given by
\be 
\t(h_0)=\sum_j{\rm arctan}(\lambda_j^0),\nonumber
\ee
where here $\lambda_j^0$ are the eigenvalues of $g^{\ell \bar q}F^0_{\bar qp}$. Thus, by possibly replacing $L$ with $L^{\otimes k}$ for a large $k$ and using the metric $h^k_0$, we can assume that the above angle is larger than $(n-1)\frac\pi2$ at all points on $X$. We drop the $k$ from our notation, and just work with an initial metric that satisfies this hypercritical phase condition. 

\begin{prop}
Suppose that $X$ has non-negative orthogonal bisectional curvature, and $L$ is a positive line bundle over $X$. If at time $t=0$ the hypercritical phase condition is satisfied, then the function $v(t):=|\zeta(t)|$ is uniformly bounded above for all time.
\end{prop}
\begin{proof}

We begin the proof by noting that the condition $\t>(n-1)\frac{\pi}2$ is preserved by the flow. To see this, observe that \eqref{thetaflow} implies
\be
\dot\t=\Delta_\eta\t.
\ee
By the maximum principle, $\inf_X\t(t)\geq \inf_X\t(0)>(n-1)\frac{\pi}2$, and thus the hypercritical phase condition is preserved. This condition is important because it implies the positivity of $iF$ along the flow. As before, let $\{\lambda_1,...,\lambda_n\}$ be the eigenvalues of the endomorphism $K^{\ell}{}_p:=g^{\ell\bar m}F_{\bar mp}$. Because the function arctan$(\cdot)$ takes values in the open interval $(-\frac\pi2,\frac\pi2)$, if $\t>(n-1)\frac{\pi}2$, then it is clear $\lambda_j>\epsilon_0$ for $1\leq j\leq n$, where $\epsilon_0$ is a fixed constant that only depends on $ \inf_X\t(h_0)$.

We now apply the heat operator to the function ${\rm log}(v)$. First, recall the computation of the time derivative of $v$, which we gave in the proof of Proposition \ref{criticalpoint}:
\be
\dot v=\eta^{j\bar k}F_{\bar kp}g^{p\bar q}\nabla_{\bar q}\nabla_j\dot\phi v=\eta^{j\bar k}F_{\bar kp}g^{p\bar q}\nabla_{\bar q}H_jv.\nonumber
\ee
As a result the time derivative of ${\rm log}(v)$ is given by:
\be
\frac{d}{dt}{\rm log}(v)=\eta^{j\bar k}F_{\bar kp}g^{p\bar q}\nabla_{\bar q} H_j=\eta^{j\bar k}F_{\bar kp}g^{p\bar q}\nabla_{\bar q} \left(\eta^{\ell\bar m}\nabla_{\ell}F_{\bar mj}\right).
\nonumber
\ee
Distributing the derivative inside the parenthesis gives:
\bea
\label{need this}
\frac{d}{dt}{\rm log}(v)&=&\eta^{j\bar k}F_{\bar kp}g^{p\bar q}\eta^{\ell\bar m}\nabla_{\bar q}\nabla_{\ell}F_{\bar mj}-\eta^{j\bar k}F_{\bar kp}g^{p\bar q}\eta^{\ell\bar s}F_{\bar sr}g^{r\bar b}\nabla_{\bar q}F_{\bar ba}\eta^{a\bar m}\nabla_{\ell}F_{\bar mj}\nonumber\\
&&-\eta^{j\bar k}F_{\bar kp}g^{p\bar q}\eta^{\ell\bar s}\nabla_{\bar q}F_{\bar sr}g^{r\bar b}F_{\bar ba}\eta^{a\bar m}\nabla_{\ell}F_{\bar mj}.
\eea
Next we compute the Laplacian of ${\rm log} (v)$. We begin by taking one space derivative:
\be
\nabla_{\bar m}{\rm log} (v)=\eta^{j\bar k}F_{\bar kp}g^{p\bar q}\nabla_{\bar m}F_{\bar qj}.\nonumber
\ee
Thus we see the Laplacian of ${\rm log} (v)$ is given by
\bea
\Delta_\eta{\rm log}(v)&=&\eta^{\ell\bar m}\nabla_\ell\left(\eta^{j\bar k}F_{\bar kp}g^{p\bar q}\nabla_{\bar m}F_{\bar qj}\right)\nonumber\\
&=&\eta^{\ell\bar m}\eta^{j\bar k}F_{\bar kp}g^{p\bar q}\nabla_\ell\nabla_{\bar m}F_{\bar qj}+\eta^{\ell\bar m}\eta^{j\bar k}\nabla_\ell F_{\bar kp}g^{p\bar q}\nabla_{\bar m}F_{\bar qj}\nonumber\\
&&-\eta^{\ell\bar m}\eta^{j\bar s}\nabla_\ell F_{\bar sr}g^{r\bar b}F_{\bar ba}\eta^{a\bar k}F_{\bar kp}g^{p\bar q}\nabla_{\bar m}F_{\bar qj}\nonumber\\
&&-\eta^{\ell\bar m}\eta^{j\bar s}F_{\bar sr}g^{r\bar b}\nabla_\ell F_{\bar ba}\eta^{a\bar k}F_{\bar kp}g^{p\bar q}\nabla_{\bar m}F_{\bar qj}.\nonumber
\eea
Applying equation \eqref{ginverse} gives the simplified expression
\bea
\Delta_\eta{\rm log}(v)&=&\eta^{\ell\bar m}\eta^{j\bar k}F_{\bar kp}g^{p\bar q}\nabla_\ell\nabla_{\bar q}F_{\bar mj}+\eta^{\ell\bar m}\eta^{j\bar k}\eta^{p\bar q}\nabla_\ell F_{\bar kp}\nabla_{\bar m}F_{\bar qj}\nonumber\\
&&-\eta^{\ell\bar m}\eta^{j\bar s}F_{\bar sr}g^{r\bar b}\nabla_\ell F_{\bar ba}\eta^{a\bar k}F_{\bar kp}g^{p\bar q}\nabla_{\bar m}F_{\bar qj}.\nonumber
\eea
Next we show that the final term on the right hand side above cancels with a term from the time derivate of ${\rm log} (v)$. To see this, first write this term in normal coordinates:
\be
\label{step}
-\eta^{j\bar j}\eta^{\ell\bar \ell}F_{\bar jj}\nabla_{\ell}F_{\bar jr}\eta^{r\bar r}F_{\bar rr}\nabla_{\bar \ell}F_{\bar rj}.
\ee
Again working in normal coordinates, we see that last term from equation \eqref{need this} is equal to
\be
-\eta^{j\bar j}\eta^{\ell\bar \ell}F_{\bar jj}\nabla_{\bar j}F_{\bar \ell r}\eta^{r\bar r}F_{\bar rr}\nabla_{ \ell}F_{\bar rj}.\nonumber
\ee
Switching $r$ and $j$ in the above expression gives
\be
-\eta^{j\bar j}\eta^{\ell\bar \ell}F_{\bar jj}\nabla_{\bar r}F_{\bar \ell j}\eta^{r\bar r}F_{\bar rr}\nabla_{ \ell}F_{\bar jr},\nonumber
\ee
and as a result this term is equal to \eqref{step} by the second Bianchi identity. Thus we have
\bea
\left(\frac{d}{dt}-\Delta_\eta\right){\rm log}(v)&=&\eta^{j\bar k}F_{\bar kp}g^{p\bar q}\eta^{\ell\bar m}[\nabla_{\bar q},\nabla_{\ell}]F_{\bar mj}-|\nabla F|^2_\eta\nonumber\\
&&-\eta^{j\bar k}F_{\bar kp}g^{p\bar q}\eta^{\ell\bar s}F_{\bar sr}g^{r\bar b}\nabla_{\bar q}F_{\bar ba}\eta^{a\bar m}\nabla_{\ell}F_{\bar mj}\nonumber
\eea
The second term on the right hand side above is clearly non-positive. In fact, the remaining two terms on the right hand side are non-positive as well. In normal coordinates, the last term above is given by
\be
-\eta^{j\bar j}F_{\bar jj}\eta^{\ell\bar\ell}F_{\bar \ell\ell}\nabla_{\bar j}F_{\bar \ell r}\eta^{r\bar r}\nabla_{\ell}F_{\bar rj}=-\sum_{\ell jr}\frac{\lambda_j\lambda_\ell|\nabla_jF_{\bar r\ell}|^2}{(1+\lambda_j^2)(1+\lambda_r)^2(1+\lambda_\ell^2)}.\nonumber
\ee
Because the eigenvalues $\lambda_\ell$ are all positive, the above expression is non-positive. We now concentrate on the commutator term:
\be
\eta^{j\bar k}F_{\bar kp}g^{p\bar q}\eta^{\ell\bar m}[\nabla_{\bar q},\nabla_{\ell}]F_{\bar mj}=\eta^{j\bar k}F_{\bar kp}g^{p\bar q}\eta^{\ell\bar m}R_{\bar q\ell}{}^r{}_jF_{\bar mr}-\eta^{j\bar k}F_{\bar kp}g^{p\bar q}\eta^{\ell\bar m}R_{\bar q\ell\bar m}{}^{\bar s}{}F_{\bar sj}.\nonumber
\ee
Again, working in normal coordinates we see
\be
\eta^{j\bar j}F_{\bar jj}\eta^{\ell\bar\ell}R_{\bar j\ell}{}^\ell{}_jF_{\bar\ell\ell}-\eta^{j\bar j}F_{\bar jj}\eta^{\ell\bar\ell}R_{\bar j\ell\bar \ell}{}^{\bar j}F_{\bar jj}=\sum_{j\ell}\left(\lambda_j\lambda_\ell-\lambda_j^2 \right)\frac{R_{\bar j\ell\bar\ell j}}{(1+\lambda_j^2)(1+\lambda_\ell^2)}.\nonumber
\ee
In the sum above, it is clear the terms with $j=\ell$ vanish. In each of the terms with $j\neq\ell$, we have $R_{\bar j\ell\bar\ell j}=R_{\bar \ell j\bar j\ell}$ since $g_{\bar kj}$ is K\"ahler. Thus we can write the above sum as
\bea
\sum_{j\ell}\frac{(\lambda_j\lambda_\ell-\lambda_j^2 )R_{\bar j\ell\bar\ell j}}{(1+\lambda_j^2)(1+\lambda_\ell^2)}&=&\sum_{j<\ell}\frac{(2\lambda_j\lambda_\ell-\lambda_j^2-\lambda_\ell^2 )R_{\bar j\ell\bar\ell j}}{(1+\lambda_j^2)(1+\lambda_\ell^2)}\nonumber\\
&=&-\sum_{j<\ell}\frac{(\lambda_j-\lambda_\ell)^2R_{\bar j\ell\bar\ell j}}{(1+\lambda_j^2)(1+\lambda_\ell^2)}\nonumber.
\eea
As a result, if $X$ has non-negative orthogonal bisectional curvature, the above expression is non-positive. Putting everything together, we have proven the following estimate
\be
\left(\frac{d}{dt}-\Delta_\eta\right){\rm log}(v)\leq 0.\nonumber
\ee
The upper bound for log$(v)$ now follows from the maximum principle. This completes the proof of the proposition.
\end{proof}

The above bound for $v$ gives a uniform upper bound for $|F|_g$, since
\be
\label{controlofFbyv}
v=|\zeta|=\sqrt{{\rm det}(1+K^2)}=\sqrt{\Pi_j(1+\lambda_j^2)}\geq\sqrt{\sum_j\lambda_j^2}=|F|_g.
\ee
Thus, for every eigenvalue of $K$ we have demonstrated the following bounds:
\be
\epsilon_0<\lambda_j<C{\qquad}1\leq j\leq n.\nonumber
\ee
 This brings us to the following important proposition.
\begin{prop}
\label{longtime}
Suppose that $h=e^{-\phi}h_0$ solves the line bundle mean curvature flow and that the curvature $F$ of $h$ is bounded above and below uniformly in time
\be
\label{C2bounds}
\epsilon_0\o\leq iF\leq C\o.
\ee
Then the flow \eqref{flow} exists for all time.
\end{prop}
\begin{proof}
Suppose not, and let $T<\infty$ denote the maximal existence time of the flow \eqref{flow}. We will prove that $\phi(t)$ is bounded in $C^k$ for any $k$ independent of $T$, and thus it converges along a subsequence to a smooth potential $\phi(T)$. We can then continue the flow beginning with the smooth metric $h_T=e^{-\phi(T)}h_0$.

First, we note that 
\be
\frac{d}{dt}\int_X\phi\,\frac{\o^n}{n!}=\int_X(\t-\hat\t)\,\frac{\o^n}{n!}\leq C.\nonumber
\ee
As a result the average value of $\phi$ is bounded for any finite time $T$. Uniform control of $iF$ implies $\phi$ in bounded in $C^0$, and thus $\phi\in C^k$ for $0\leq k\leq 2$. We now prove higher order bounds for $\phi$.

Note that our assumption \eqref{C2bounds} implies that the Laplacian $\Delta_\eta=\eta^{p\bar q}\pl_p\pl_{\bar q}$ is uniformly elliptic in time and space. As we have see, because $\t$ solves
\be
\left(\frac d{dt}-\Delta_\eta\right)\t=0,\nonumber
\ee
the function $\t$ is bounded above by the maximum principle. Furthermore, by the linear parabolic estimates of Krylov-Safonov \cite{KS} (see also \cite{Kr}, Theorem 11, Section 4.2), uniform ellipticity of $\eta^{p\bar q}$ implies that there exists an $\al$ so that $\t$ is bounded in $C^\al$ in space and $C^{\frac\al2}$ in time. Because $\t=\dot\phi$, we in fact have that $\phi$ is uniformly bounded in $C^{1,\frac\al 2}$ in time.

Now, recall our definition of $\t$, given by equation \eqref{theta}. We rewrite this definition below in terms of a uniformly elliptic operator $Q$:
\be
\label{operatorQ}
Q(\pl\bar\pl \phi):=-i{\rm log}\frac{{\rm det}(I-iK)}{\sqrt{{\rm det}(I+K^2)}}=\t.
\ee
As we demonstrated in Section \ref{volume}, the operator $Q$ is readily seen to be elliptic by equation \eqref{evolvet}. In fact, assumption \eqref{C2bounds} implies $Q$ is concave as well. The proof of this follows a similar computation to one found in \cite{S2,SW}, which we include here for the reader's convenience. Consider the path of functions $\phi+\epsilon u$ for any smooth function $u\in C^\infty(X)$. Then by \eqref{evolvet} we have
\be
\frac{d}{d\epsilon}Q(\pl\bar\pl(\phi+\epsilon u))=\eta^{p\bar q}u_{\bar qp}.\nonumber
\ee
Taking a second derivative gives
\be
\left(\frac{d}{d\epsilon}\right)^2Q(\pl\bar\pl(\phi+\epsilon u))=-\eta^{p\bar m}u_{\bar mj}g^{j\bar k}F_{\bar k\ell}\eta^{\ell\bar q}u_{\bar qp}-\eta^{p\bar m}F_{\bar mj}g^{j\bar k}u_{\bar k\ell}\eta^{\ell\bar q} u_{\bar qp}.\nonumber
\ee
Writing the above expression in normal coordinates we see
\be
\left(\frac{d}{d\epsilon}\right)^2Q(\pl\bar\pl(\phi+\epsilon u))=-2\sum_{pj}\frac{\lambda_j\,|u_{\bar pj}|^2}{(1+\lambda_j^2)(1+\lambda_p^2)}.\nonumber
\ee
Thus as long as $\epsilon_0\o\leq iF$, the operator $Q$ is concave.

At this point we have shown $Q$ is a uniformly elliptic, concave operator. Because the right hand side of \eqref{operatorQ} is uniformly bounded in $C^\al$, we would like to apply the Evans-Krylov theorem \cite{E,Kr2} to show $\phi$ is in fact uniformly bounded in $C^{2,\b}$ for some $\b\in(0,1)$. However, we only have uniform bounds for the complex Hessian $\pl\bar\pl\phi$, and not the real Hessian $D^2\phi$, so we can not apply the Evans-Krylov theorem directly. To achieve the bounds, one can apply the trick of Wang \cite{WYu}, which was later extended to more general settings by Tosatti-Wang-Weinkove-Yang \cite{TWWY}. The key idea of Wang is to include the space of $n\times n$ Hermitian matrices into the subspace of $2n$ dimensional real symmetric matrices preserved by the complex structure, and then show $\phi$ satisfies a modified PDE (which is strictly uniformly elliptic in the real sense). Wang then proves his estimate by applying an improvement of the Evans-Krylov theorem due to Caffarelli \cite{C}. Because our operator $Q$ is a smooth function of the eigenvalues of   $K$, our setup fits well into the formalism defined in \cite{TWWY}, and thus we can apply their main result to conclude $C^{2,\beta}$ bounds for $\phi$. We direct the reader to \cite{TWWY} for details.

Choose a constant $\gamma\in(0,1)$ such that $\gamma<\,$min$\{\beta,\frac\al2\}$. So far we have shown $\phi$ is uniformly bounded in $C^{2,\gamma}$ in space and $C^{1,\gamma}$ in time. We can now prove higher order bounds. Define the function $\psi:=\pl_j\phi$ for any fixed $j$. Taking the time derivative of $\psi$ gives
\be
\dot\psi=H_j=\eta^{p\bar q}\nabla_pF_{\bar qj}=\eta^{p\bar q}\nabla_p\pl_{\bar q}\pl_j\phi+\eta^{p\bar q}\nabla_pF^0_{\bar qj}.\nonumber
\ee 
Thus we have
\be
\left(\frac d{dt}-\Delta_\eta\right)\psi=\eta^{p\bar q}\nabla_pF^0_{\bar qj}.\nonumber
\ee
The right hand side of the above equation is uniformly bounded in $C^{\gamma}$. Furthermore, the coefficients $\eta^{p\bar q}$ of $\Delta_\eta$ are controlled in $C^\gamma$ as well. Therefore, by standard Schauder estimates, $\psi\in C^{2,\gamma}$. Thus $\phi\in C^{3,\gamma}$, and the coefficients $\eta^{p\bar q}$ are controlled in $C^{1,\gamma}$. We can continue to bootstrap in this fashion to get all higher order bounds for $\phi$.
\end{proof}

We remark that this proposition relies heavily on the assumption that $iF$ remains a K\"ahler form. In general, when the line bundle $L$ is not positive, we do not expect $C^2$ bounds for $\phi$ to imply higher order estimates. We now prove Theorem \ref{bisectional}.
\begin{proof}[Proof of Theorem \ref{bisectional}]

Up to this point we have established long time existence of the flow \eqref{flow}, under the assumptions that $X$ has non-negative orthogonal bisectional curvature and $\t$ satisfies the hypercritical phase condition. Using long time existence, we construct a sequence of metrics which converge smoothly to a metric of constant mean curvature. 

First, we integrate the $L^2$ norm of $H$ in time. By Proposition \ref{criticalpoint} and the positivity $V(\cdot)$, we see this integral is bounded as $t$ approaches infinity:
\be
\int_0^\infty\int_X|H|^2_\eta \,v\,\frac{\o^n}{n!}\, dt = \lim_{a\ra\infty} \left(V(0)-V(a)\right)<\infty.\nonumber
\ee
As a result there exists a subsequence of times $t_j$ so that $||H(t_j)||_{L^2(\eta)}$ converges to zero as $j$ approaches infinity. Let $h_j=e^{-\phi(t_j)}h_0$ be the corresponding sequence of metrics. 

We next consider the following sequence of functions
\be
\psi_j=\phi(t_j)-\int_X\phi(t_j)\,\frac{\o^n}{n!},\nonumber
\ee
and define the sequence of metrics $\ti h_j=e^{-\psi_j}h_0$. Clearly both $\ti h_j$ and $h_j$ define the same curvatures, thus by \eqref{C2bounds} both the first and second derivatives of $\psi_j$ are bounded. Since $\psi_j$ is normalized to have zero average, this implies that $\psi_j$ is bounded in $C^0$ as well. Thus the $\psi_j$ are uniformly bounded in $C^2$. Furthermore, using the operator $Q$ from \eqref{operatorQ} we know
\be
Q(\pl\bar\pl\psi_j)=\t.\nonumber
\ee
Just as in the previous proposition, $\t$ is uniformly bounded in $C^\al$, thus we have uniform $C^{2,\al}$ bounds for $\psi_j$, which we can then bootstrap up to get higher order derivative bounds for $\psi_j$. 

Higher order bounds for $\psi_j$ imply there exists a subsequence (still denoted $\psi_j$) that converges to a limiting function $\psi_\infty$ in $C^\infty$, which defines a smooth limiting metric $h_\infty=e^{-\psi_\infty}h_0$ on $L$. Furthermore, because both $h_j$ and $\ti h_j$ define the same curvature $F_j$, our choice of subsequence implies the second fundamental forms $H(t_j)$ of $\ti h_j$ converge to zero in $L^2$. Thus $H(h_\infty)=0$, and as a result $h_\infty$ solves \eqref{DHYMn}. This completes the proof of the theorem.
 \end{proof}

 We conclude this section with a proof of Theorem \ref{higherbounds}, which shows that all higher order bounds for the curvature $F$ follow from a bound on the first derivative $\nabla F$. It will be useful extend the Laplacian $\Delta_\eta$ from $C^\infty(X)$ to the spaces $\Lambda^{p,q}(X)$.  Using the covariant derivative $\nabla$, we construct the following operators
\be
\Delta_\eta:=\eta^{p\bar q}\nabla_p\nabla_{\bar q}\qquad{\rm and}\qquad\bar\Delta_\eta:=\eta^{p\bar q}\nabla_{\bar q}\nabla_p.
\ee
We point out that these operators are not particularly well behaved, primarily for the reason that $\nabla$ is the covariant derivative with respect to $g_{\bar kj}$ as opposed to $\eta_{\bar kj}$. However, the covariant derivative $\nabla$ does appear in equation \eqref{mcloc}, and as a result the above operators show up quite naturally in our computations.

 \begin{proof}[Proof of Theorem \ref{higherbounds}]

Because the functional $V(h_t)$ is decreasing along the flow, the function $v$ is uniformly bounded in $L^1$, and thus $|F|_g$ is integrable by equation \eqref{controlofFbyv}. It follows that $\inf_X|F|_g$ is bounded uniformly in time, and we can apply assumption \eqref{assump1} to conclude that $\sup_X|F|_g$ is bounded as well.

As we have seen, the $C^0$ bound for $F$ implies the metrics $g_{\bar kj}$ and $\eta_{\bar kj}$ are equivalent, and in particular we know that $\eta^{j\bar k}$ does not degenerate along the flow. Therefore, we can use either metric to compute the norms of our evolving quantities. For this theorem it is convenient to use the evolving metric $\eta_{\bar kj}$, and we use this metric exclusively for the rest of the proof. To ease notation we simply denote $|\cdot|_\eta$ by $|\cdot|$.

We prove the higher order derivative bounds for $F$ by applying the maximum principle. To accomplish this, we need to compute the evolution equations for several key terms along the flow, begging with $F_{\bar kj}$. Equations \eqref{mcloc} and \eqref{evolve1} give:
\be
\dot F_{\bar kj}=\nabla_{\bar k}\nabla_j\dot\phi=\nabla_{\bar k}H_j=\nabla_{\bar k}(\eta^{p\bar q}\nabla_{j}F_{\bar qp}).\nonumber
\ee
 Distributing the derivative we get
 \be
\dot F_{\bar kj}=\eta^{p\bar q}\nabla_{\bar k}\nabla_jF_{\bar qp}+\nabla_{\bar k}(\eta^{p\bar q})\nabla_jF_{\bar qp}.\nonumber
 \ee
We apply the second Bianchi identity to the first term on the right, and commute derivatives as follows:
\bea
\eta^{p\bar q}\nabla_{\bar k}\nabla_pF_{\bar qj}&=&\eta^{p\bar q}\nabla_{\bar k}\nabla_{p}F_{\bar qj}\nonumber\\
&=&\eta^{p\bar q}[\nabla_{\bar k},\nabla_p]F_{\bar qj}+\eta^{p\bar q}\nabla_{p}\nabla_{\bar k}F_{\bar qj}\nonumber\\
&=&\eta^{p\bar q}R_{\bar kp}{}^\ell{}_jF_{\bar q\ell}-\eta^{p\bar q}R_{\bar kp}{}^{\bar m}{}_{\bar q}F_{\bar mj}+\Delta_\eta F_{\bar kj}.\nonumber
\eea
Next we compute out the derivate of $\eta^{j\bar k}$.
\be
\nabla_{\bar k}(\eta^{p\bar q})\nabla_jF_{\bar qp}=-\eta^{p\bar m}\nabla_{\bar k}F_{\bar mr}g^{r\bar s}F_{\bar s\ell}\eta^{\ell\bar q}\nabla_j F_{\bar qp}-\eta^{p\bar m}F_{\bar mr}g^{r\bar s}\nabla_{\bar k}F_{\bar s\ell}\eta^{\ell\bar q}\nabla_j F_{\bar qp}.\nonumber
\ee
Putting everything together we see
\bea
\label{evolveF}
\dot F_{\bar kj}&=&\Delta_\eta F_{\bar kj}+\eta^{p\bar q}R_{\bar kp}{}^\ell{}_jF_{\bar q\ell}-\eta^{p\bar q}R_{\bar kp}{}^{\bar m}{}_{\bar q}F_{\bar mj}-\eta^{p\bar m}\nabla_{\bar k}F_{\bar mr}g^{r\bar s}F_{\bar s\ell}\eta^{\ell\bar q}\nabla_j F_{\bar qp}\nonumber\\
&&-\eta^{p\bar m}F_{\bar mr}g^{r\bar s}\nabla_{\bar k}F_{\bar s\ell}\eta^{\ell\bar q}\nabla_j F_{\bar qp}.
\eea
This gives the evolution of $F$.

We now turn to the evolution of first derivative of $F$, so we take the derivative of the above expression. We will utilize the notation $A*B$ for any combination of the tensors $A$ and $B$ where the exact form is not necessary.
\bea
\frac{d}{dt}( \nabla F )&=&\nabla \Delta_\eta F+\nabla(\eta^{-1}*R*F)+\nabla(\eta^{-2}*F*\nabla F*\bar\nabla F)\nonumber.
\eea
Interchanging the order of the derivative $\nabla$ and the Laplacian $\Delta_\eta$ gives
\bea
\nabla_p \Delta_\eta F_{\bar kj}&=&\nabla_p\eta^{r\bar s}\nabla_r\nabla_{\bar s} F_{\bar kj}+\eta^{r\bar s}\nabla_r\nabla_p\nabla_{\bar s} F_{\bar kj}\nonumber\\
&=&\nabla_p\eta^{r\bar s}\nabla_r\nabla_{\bar s} F_{\bar kj}-\eta^{r\bar s}\nabla_r\left([\nabla_{\bar s},\nabla_{p}] F_{\bar kj}\right)+\Delta_\eta\nabla_p F_{\bar kj}.\nonumber
\eea
Thus we have
\bea
\label{deriv}
\frac{d}{dt}( \nabla F )&=&\Delta_\eta\nabla F+\eta^{-2}*F*\nabla F*\nabla\bar\nabla F+\eta^{-1}\nabla R*F+\eta^{-1}R*\nabla F\nonumber\\
&&+\nabla(\eta^{-1}*R*F)+\nabla(\eta^{-2}*F*\nabla F*\bar\nabla F).
\eea
We will come back to this above equation in order to arrive at the general form for higher order derivatives of $F$.

At this point we are able to compute bounds for the heat operator on $|\nabla F|^2$, beginning with the time derivative. Using the fact that $F$ is bounded in $C^0$, the above equation gives
\bea
\langle \nabla\dot F, \nabla F\rangle&\leq& \langle \Delta_\eta\nabla F, \nabla F\rangle+C\left(|\nabla\bar\nabla F||\nabla F|^2+|\nabla\nabla F||\nabla F|^2\right)\nonumber\\
&&+C\left(|\nabla F|+|\nabla F|^2+|\nabla F|^4\right).\nonumber
\eea
Furthermore, we can apply our main assumption \eqref{assump1} to conclude
\be
\label{thing1}
\langle \nabla\dot F, \nabla F\rangle \leq \langle \Delta_\eta\nabla F, \nabla F\rangle+C\left(|\nabla\bar\nabla F|+|\nabla\nabla F|\right)+C.
\ee
Here we note that because all of are norms are computed using the evolving metric $\eta_{\bar kj}$, taking the time derivative of $|\nabla F|^2$ gives
\be
\label{thing2}
\frac{d}{dt}|\nabla F|^2\leq2\langle \nabla\dot F, \nabla F\rangle+3|\dot\eta^{-1}||\nabla F|^2.
\ee
The second term on the right needs to be estimated. We compute out the time derivative of $\eta^{j\bar k}$
\bea
\dot\eta^{p\bar q}&=&-\eta^{p\bar m}\dot F_{\bar m\ell}g^{\ell \bar k}F_{\bar kj}\eta^{j\bar q}-\eta^{p\bar m} F_{\bar m\ell}g^{\ell \bar k}\dot F_{\bar kj}\eta^{j\bar q}.\nonumber
\eea
Again using our $C^0$ bound for $F$, we see that
\be
|\dot\eta^{-1}|\leq C |\dot F|\leq C|\nabla\bar\nabla F|+|\nabla F|^2+C.\nonumber
\ee
where the second inequality follows from applying equation \eqref{evolveF}. Combining the above inequality with \eqref{thing1} and \eqref{thing2} gives:
\bea
\frac{d}{dt}|\nabla F|^2&\leq& 2\langle \Delta_\eta\nabla F, \nabla F\rangle+C\left(|\nabla\bar\nabla F|+|\nabla\nabla F|\right)+C\nonumber\\
&\leq&\langle \Delta_\eta\nabla F, \nabla F\rangle+\langle \nabla F, \bar\Delta_\eta\nabla F\rangle+\frac18|\nabla\bar\nabla F|^2+\frac18|\nabla\nabla F|^2+C,\nonumber
\eea
where for the second inequality  we used Young's inequality, and the fact that interchanging $\Delta_\eta\nabla F$ with $\bar\Delta_\eta\nabla F$ introduces curvature terms on $X$, which are fixed and bounded.

We now apply the Laplacian $\Delta_\eta$ to $|\nabla F|^2$. Because the norm $|\cdot|$ is computed using the metric $\eta_{\bar kj}$, and our covariant derivative $\nabla$ is with respect to the metric $g_{\bar kj}$, we will get some extra terms when the derivatives land on the metric. However, these terms are easily controlled. Note that
\bea
\nabla_r\eta^{p\bar q}&=&-\eta^{p\bar m}\nabla_r F_{\bar m\ell}g^{\ell \bar k}F_{\bar kj}\eta^{j\bar q}-\eta^{p\bar m} F_{\bar m\ell}g^{\ell \bar k}\nabla_rF_{\bar kj}\eta^{j\bar q},\nonumber
\eea
so if at most one derivative lands on a metric term, it is bounded by $|\nabla F|$ and thus controlled in $C^0$. If two derivatives land on $\eta_{\bar kj}$, it is possible to create a second order derivative of $F$. Keeping this in mind we see
\bea
\Delta_\eta|\nabla F|^2&\geq& \langle \Delta_\eta\nabla F, \nabla F\rangle+\langle \nabla F, \bar\Delta_\eta\nabla F\rangle+|\nabla\nabla F|^2+|\bar\nabla\nabla F|^2\nonumber\\
&&-C|\nabla\nabla F|-C|\bar\nabla\nabla F|-C.\nonumber
\eea
Putting everything together, and applying Young's inequality to the second derivative terms above, we conclude 
\be
\label{dfbound}
\left(\frac d{dt}-\Delta_\eta\right)|\nabla F|^2\leq -\frac34\left(|\nabla\bar\nabla F|^2+|\nabla\nabla F|^2\right)+C.
\ee
This is the key inequality we will use when applying the heat operator to $|\nabla F|^2$.

We now turn to the second order derivatives of $F$. We began by taking the derivative of \eqref{deriv}, which gives the following equation:
\bea
\label{bigderiv}
\frac{d}{dt}(\nabla\nabla F )&=&\nabla\Delta_\eta\nabla F+\nabla(\eta^{-2}*F*\nabla F*\nabla\bar\nabla F+\eta^{-1}\nabla R*F)\nonumber\\
&&+\nabla(\eta^{-1}R*\nabla F)+\nabla\nabla(\eta^{-1}*R*F+\eta^{-2}*F*\nabla F*\bar\nabla F).
\eea
We want to interchange the order of the derivative and Laplacian in the first term above. As before we see
\bea
\nabla_p \Delta_\eta \nabla_{q} F_{\bar kj}&=&\nabla_p\eta^{r\bar s}\nabla_r\nabla_{\bar s} \nabla_{ q}F_{\bar kj}+\eta^{r\bar s}\nabla_r\nabla_p\nabla_{\bar s} \nabla_{ q}F_{\bar kj}\nonumber\\
&=&\nabla_p\eta^{r\bar s}\nabla_r\nabla_{\bar s} \nabla_{ q}F_{\bar kj}-\eta^{r\bar s}\nabla_r\left([\nabla_{\bar s},\nabla_{p}] \nabla_{ q}F_{\bar kj}\right)+\Delta_\eta\nabla_p\nabla_{q} F_{\bar kj}.\nonumber
\eea
Combining the previous two equations, and using the fact that $F$, and $\nabla F$ are all bounded in $C^0$, we arrive at the following bound:
\bea
\langle \nabla\nabla \dot F, \nabla\nabla F\rangle&\leq& \langle \Delta_\eta\nabla\bar\nabla F, \nabla\bar\nabla F\rangle+C(|\nabla\bar\nabla\nabla F||\nabla\nabla F|+|\nabla\nabla\nabla F||\nabla\nabla F|)\nonumber\\
&&+C(|\nabla\nabla F|^2+|\nabla\bar\nabla F||\nabla\nabla F|+|\nabla\bar\nabla F||\nabla\nabla F|^2).\nonumber
\eea
Again, when we take the time derivative of $|\nabla\nabla F|^2$, we get a contribution from the metric terms. However, as we have seen, $\dot\eta^{j\bar k}$ is controlled by $|\nabla\bar\nabla F|$, and thus these terms can be absorbed into $C|\nabla\bar\nabla F||\nabla\nabla F|^2$. Furthermore, when taking the Laplacian $\Delta_\eta$ of $|\nabla\nabla F|^2$, we get extra terms from when the derivatives hit $\eta^{j\bar k}$, but once again this terms can be absorbed into $C|\nabla\bar\nabla F||\nabla\nabla F|^2$, since at most two derivatives can hit $\eta^{j\bar k}$. Finally, we assume without loss of generality that $|\nabla\nabla F|^2$ is bigger than one. This leads to the following important inequality
\bea
\label{ddfbound}
\left(\frac d{dt}-\Delta_\eta\right)|\nabla\nabla F|^2&\leq& -\frac34|\nabla\nabla\nabla F|^2-\frac34|\bar\nabla\nabla\nabla F|^2\nonumber\\
&&+C|\nabla\bar\nabla F||\nabla\nabla F|^2.
\eea
Following the exact same method, we can also prove an identical bound for the mixed partial derivatives $|\nabla\bar\nabla F|^2$.

We are now ready to prove a $C^0$ bound for the second order derivatives of $F$. For the argument to follow, we find it much simpler to change our notation for derivatives, avoiding the cumbersome use of both barred and unbarred derivatives. Let $D=\nabla+\bar\nabla$ denote all possible first order derivatives, while $D^kF$ denotes all possible derivatives of $F$ to $k$-th order (both barred and unbarred). Then inequality \eqref{dfbound} can be rewritten as
\be
\label{dfbound2}
\left(\frac d{dt}-\Delta_\eta\right)|D F|^2\leq -\frac34|D^2F|^2+C.
\ee
Furthermore, both \eqref{ddfbound} and the identical computation for $|\nabla\bar\nabla F|^2$ can be combined to give
\bea
\label{ddfbound2}
\left(\frac d{dt}-\Delta_\eta\right)|D^2F|^2&\leq&-\frac34|D^3F|^2+C|D^2F|^3.
\eea
We argue as follows. Let $A\geq1$ denote the $C^0$ bound for $|DF|^2.$ For some large constant $C_0$ (to be determined), define the function $f=(C_0A+|DF|^2)|D^2F|^2$. The time derivative of $f$ is given by
\be
\dot f=\frac{d}{dt}|DF|^2|D^2F|^2+(C_0A+|DF|^2)\frac{d}{dt}|D^2F|^2.\nonumber
\ee
Furthermore, taking Laplacian $\Delta_\eta$ of $f$ gives:
\be
\Delta_\eta f\geq\Delta_\eta|DF|^2|D^2F|^2+(C_0A+|DF|^2)\Delta_\eta|D^2F|^2-8A|D^2F|^2|D^3F|.\nonumber
\ee
This allows us to conclude 
\bea
\left(\frac d{dt}-\Delta_\eta\right)f&\leq& \left(\frac d{dt}-\Delta_\eta\right)|DF|^2|D^2F|^2+\frac14|D^2F|^4+64A^2|D^3F|^2\nonumber\\
&&+(C_0A+|DF|^2)\left(\frac d{dt}-\Delta_\eta\right)|D^2F|^2\nonumber\\
&\leq&-\frac12|D^2F|^4+C|D^2F|^2+64A^2|D^3F|^2\nonumber\\
&&-\frac34(C_0A+|DF|^2)|D^3F|^2+C(C_0A+|DF|^2)|D^2F|^3.\nonumber
\eea
Choose $C_0$ large enough so that $\frac34C_0\geq 64A$. Now, we apply Young's inequality to the positive $|D^2F|^3$ term above:
\be
C(C_0A+|DF|^2)|D^2F|^3\leq \varepsilon_1^{\frac43} |D^2F|^4+\frac1{\varepsilon_1^4}C^4(C_0A+|DF|^2)^4.\nonumber
\ee
We can also estimate the $C|D^2F|^2$ using Young's inequality
\be
C|D^2F|^2\leq \varepsilon_2|D^2F|^4+\frac1{\varepsilon_2}C^2.\nonumber
\ee
Thus, for $\varepsilon_1$ and $\varepsilon_2$ small enough, there exists a constant $C$ so that
\be
\left(\frac d{dt}-\Delta_\eta\right)f\leq -\frac14|D^2F|^4+CA^4.\nonumber
\ee
Now, by making the constant $C$ larger if necessary, the above inequality can be written with $f^2$ on the right hand side
\be
\left(\frac d{dt}-\Delta_\eta\right)f \leq-\frac{f^2}{CA^4}+CA^4.\nonumber
\ee
For the same constant $C$, we define the following function:
\be
\hat f=\frac{f}{CA^4}-A.\nonumber
\ee
Now, if $\hat f\leq 0$ for all time, than we would know $|D^2F|$ is bounded uniformly above by a constant. If not, there exists a point $p\in X$ and a time $t\in[0,T)$ such that $\hat f(p,t)>\epsilon$ for some small $\epsilon>0$. Without loss of generality we can assume that the function $\hat f$ achieves a local max at $p$, and that $f$ is nondecreasing in time.
Then at $(p,t)$ we have
\be
0\leq\left(\frac{d}{dt}-\Delta_\eta \right)\hat f\leq -\frac{f^2}{(CA^4)^2}+1.\nonumber
\ee
Note by assumption $A\geq1$. Thus we have
\be
0\leq\left(\frac{d}{dt}-\Delta_\eta \right)\hat f\leq -(\hat f+A)^2+A^2\leq-\hat f^2\leq -\epsilon^2,\nonumber
\ee
a contradiction. Thus $\hat f\leq 0$ for all time, and as a result $|D^2F|^2\leq CA^5.$ This gives us control of all second order derivates of $F$.

We now address the higher order derivative estimates. Since we have just shown that $|D^2  F|$ is bounded in $C^0$, we can rewrite \eqref{ddfbound2} as
\be
\label{ddfboundw}
\left(\frac d{dt}-\Delta_\eta\right)|D^2 F|^2\leq -\frac12|D^3  F|^2+C_2|D^2  F|^2.\nonumber
\ee
Furthermore, if one assumes that $|D^{k-1}F|^2$ is uniformly bounded for $k\geq 3$, by taking more derivatives of \eqref{bigderiv} and applying similar bounds to before, it is easy to see that
\be
\label{derivk}
\left(\frac d{dt}-\Delta_\eta\right)|D^k F|^2\leq -\frac12|D^{k+1}F|^2+C_k|D^k F|^2,\nonumber
\ee
where all the constants $C_k$ are independent of time. The higher derivative bounds follow by induction, using the exact method we used to achieve the second order bound. Specifically, assume that for $k\geq 3$ we know $|D^{k-1} F|^2\leq A_k$ for a large constant $A_k$ independent of time. Define the function $f_k=(C_0A^2_k+|D^{k-1}F|^2)|D^{k}F|^2$, and apply the heat operator. Just as before one can prove $|D^kF|^2$ is uniformly bounded by a large constant $A_{k+1}$.

This completes the proof of \eqref{goal1}. To prove convergence, we can follow the same arguments as those given in Proposition \ref{longtime} and the proof of Theorem \ref{bisectional}. \end{proof}

 \newpage
 \thispagestyle{plain}

\end{normalsize}

\begin{thebibliography}{4}





{\small


\bibitem{C} L.A. Caffarelli, {\it Interior a priori estimates for solutions of fully nonlinear equations}, Ann. Math. (2) 130(1) (1989), 189-213.

\bibitem{CLT} J. Chen, J. Li, and G. Tian, {\it Two-dimensional graphs moving by mean curvature flow,} Acta Math. Sin. 18 (2002), 209-224. 

\bibitem{CWY} J.-Y. Chen, M. Warren and Y. Yuan {\it A priori estimate for convex solutions to special Lagrangian equations and its application,}  Comm. Pure Appl. Math., 62 (2009), 583-595.
\bibitem{E} L.C. Evans, {\it Classical solutions of fully nonlinear, convex, second order ellpitic equations}, Comm. Pure Appl. Math 25 (1982), 333-363.

\bibitem{CT} T.C. Collins and V. Tosatti, {\it K\"ahler currents and null loci} (preprint) arXiv:1304.5216.

\bibitem{DP} J.-P. Demailly and M. Paun, {\it Numerical characterization of the K\"ahler cone of a compact 
K\"ahler manifold}, Ann. of Math., 159 (2004), no. 3, 1247-1274. 


\bibitem{FLM} H. Fang, M. Lai, and X. Ma, {\it On a class of fully nonlinear flows in K\"ahler geometry}, J. Reine Angew. Math. 653 (2011), 189-220.

\bibitem{GZ} H.L. Gu and Z.H. Zhang {\it An extension of Mok's theorem on the generalized Frankel conjecture}, Sci. China Math. 53 (2010), no. 5, 1253-1264.
%\bibitem{BSVY} B. Greene,  A. Shapere, C. Vafa, S.-T. Yau, {\em Stringy cosmic strings and noncompact Calabi-Yau manifolds}, Nuclear Phys. B {{} 337} (1990), no. 1, 1-36. 

\bibitem{GW} M. Gross, P. M. H. Wilson, {\em Large complex structure limits of $K3$ surfaces}, J. Differential Geom. {{} 55}, (2000), no. 3, 475-546.

\bibitem{HL}  R. Harvey and H. B. Lawson, {\it Calibrated geometries}, Acta Math. 148 (1982), 47-157.

\bibitem{H} G. Huisken, {\it Asymptotic behavior for singularities of the mean curvature flow}, J. Differential Geom. 31 (1990), no. 1, 285-299.

\bibitem{K} M. Kontsevich, {\it Homological algebra of Mirror Symmetry}, Proceedings of the International Congress of Mathematicians, Zuerich 1994, vol. I, Birkhauser 1995, 120-139.

\bibitem{KS} M. Kontsevich, Y. Soibelman, {\em Homological mirror symmetric and torus fibrations}, Symplectic geometry and mirror symmetry (Seoul, 2000), 203- 263, World Sci. Publ., River Edge, NJ, 2001.


\bibitem{Kr2} N.V. Krylov, {\it Boundedly nonhomogeneous elliptic and parabolic equations}, Izvestia 
Akad. Nauk. SSSR 46 (1982), 487-523; English translation in Math. USSR Izv. 20 
(1983), no. 3, 459-492. 


\bibitem{Kr} N.V. Krylov, {\it Nonlinear elliptic and parabolic equations of the second order}. Mathematics and its Applications (Soviet Series), 7. D. Reidel Publishing Co., Dordrecht, 1987. xiv+462 pp. ISBN: 90-277-2289-7. 

\bibitem{Kr3} N.V. Krylov, {\it On the general notion of fully nonlinear second-order elliptic equations}, Trans. Amer. Math. Soc. 347 (1995), 857-895.


\bibitem{KS} N.V. Krylov and M.V. Safonov, {\it A property of the solutions of parabolic 
equations with measurable coefficients}. (Russian) Izv. Akad. Nauk SSSR 
Ser. Mat. 44 (1980), 161-175. 

\bibitem{LYZ} C. Leung, S.-T. Yau, and E. Zaslow, {\it
From special Lagrangian to Hermitian-Yang-Mills via Fourier-Mukai transform,} Winter School on Mirror Symmetry, Vector Bundles and Lagrangian Submanifolds (Cambridge, MA, 1999), 209-225, AMS/IP Stud. Adv. Math., 23, Amer. Math. Soc., Providence, RI, 2001. 


\bibitem{MMMS} M. Marino, R. Minasian, G. Moore, and A. Strominger, {\it Nonlinear Instantons from Supersymmetric p-Branes,} hep-th/9911206. 

\bibitem{N} A. Neves, {\it Finite time singularities for Lagrangian mean curvature flow,} Ann. of Math. (2) 177 (2013), no. 3, 1029-1076.

\bibitem{P1} V. Pingali, {\it A generalized Monge-Amp\`ere equation}, (preprint) arXiv:1205.1266.

\bibitem{P2} V. Pingali, {\it A fully nonlinear ``Generalized Monge-Amp\`ere" PDE on a torus}, (preprint) arXiv:1310.1656.

%\bibitem {SSY} R. Schoen, L. Simon and S.-T. Yau, {\it Curvature estimates for minimal hypersurfaces}, Acta Math. 134 (1975), no. 3-4, 275-288.

\bibitem{SW} R. Schoen and J. Wolfson, {\it Minimizing area among Lagrangian surfaces: the mapping problem,} J. Differential Geom. 58 (2001), no. 1, 1-86.



\bibitem{S1} K. Smoczyk, {\it Angle theorems for the Lagrangian mean curvature flow}, Math. Z. 240 (2002), 849-883.

\bibitem{S2} K. Smoczyk, {\it Longtime existence of the Lagrangian mean curvature flow,} Calc. Var. Partial Differential Equations 20 (2004), 25-46. 




%\bibitem{SWZ} K. Smoczyk, G. Wang and Y. Zhang, {\it The Sasaki-Ricci flow,} Inter. J. of Math., {{} 21} (2010), no. 7, 951-969.

\bibitem{SW2} K. Smoczyk and M.-T. Wang, {\it Mean curvature flows of Lagrangians submanifolds 
with convex potentials,} J. Differential Geom. 62 (2002), 243-257. 

\bibitem{SYZ} A. Strominger, S.-T. Yau, E. Zaslow, {\em Mirror Symmetry is T-duality}, Nucl. Phys. B, {\bf 479} (1996), no. 1-2, 243-259.

\bibitem{TY} R. P. Thomas and S.-T. Yau, {\it Special Lagrangians, stable bundles and mean curvature
flow.} Comm. Anal. Geom. 10 (2002), 1075-1113.
 
 
\bibitem{TWWY} V. Tosatti, Y. Wang, B. Weinkove and X. Yang, {\it $C^{2,\al}$ estimates for nonlinear elliptic equations in complex and almost complex geometry} Calc. Var. Partial Differential Equations (to appear).
\bibitem{TW} M.-P. Tsui and M.-T. Wang, {\it Mean curvature flows and isotopy of maps between 
spheres}, Comm. Pure Appl. Math. 57 (2004), 1110-1126. 

 
\bibitem{WY1} D.-K. Wang and Y. Yuan, {\it Hessian estimates for special Lagrangian equations with critical and supercritical phases in general dimensions},   Amer. J. Math., 136 (2014), no.2, 481-499.  
 
 \bibitem{WY2} D.-K. Wang and Y. Yuan, {\it Singular solutions to special Lagrangian equations with subcritical phases and minimal surface systems,}   Amer. J. Math., 135 (2013) 1157-1177.
 
\bibitem{W} M.-T. Wang, {\it Long-time existence and convergence of graphic mean curvature flow in arbitrary codimension}, Invent. math. 148 (2002) 3, 525-543.
 
\bibitem{WYu} Y. Wang, {\it On the $C^{2,\al}$ regularity of the complex Monge-Amp\`ere equation}, Math. Res. Lett. 19 (2012), no. 4, 939-946. 


\bibitem{Y} S.-T. Yau, {\em On the Ricci curvature of a compact K\"ahler
manifold and the complex Monge-Amp\`ere equation I}, Comm. Pure Appl.
Math. {\bf 31} (1978) 339-411.
 
 \thispagestyle{plain}

}
\end{thebibliography}
\end{document}